\documentclass[12pt]{amsart}
\usepackage[T1]{fontenc}
\usepackage[utf8]{inputenc}
\usepackage{lmodern}
\usepackage{microtype}
\usepackage[margin=27mm,headheight=15pt]{geometry}
\usepackage{amsmath,amssymb,amsthm,mathtools}
\usepackage{enumitem}
\usepackage{booktabs,array}
\usepackage{fancyhdr}
\usepackage[hidelinks,pdfencoding=auto,psdextra]{hyperref}
\usepackage{bookmark}
\usepackage{tikz}
\usetikzlibrary{arrows.meta,calc,decorations.pathreplacing}
\usepackage[font=small,labelfont=bf]{caption}
\definecolor{curveblue}{RGB}{34,76,112}
\definecolor{orbitred}{RGB}{155,58,46}
\setlist[enumerate]{itemsep=3pt,topsep=5pt}
\setlist[itemize]{itemsep=3pt,topsep=5pt}
\numberwithin{equation}{section}
\newtheorem{theorem}{Theorem}[section]
\newtheorem{proposition}[theorem]{Proposition}
\newtheorem{lemma}[theorem]{Lemma}
\newtheorem{corollary}[theorem]{Corollary}
\theoremstyle{definition}
\newtheorem{definition}[theorem]{Definition}
\newtheorem{hypothesis}[theorem]{Theorem}
\theoremstyle{remark}
\newtheorem{remark}[theorem]{Remark}
\newcommand{\C}{\mathbb C}
\newcommand{\R}{\mathbb R}
\newcommand{\Z}{\mathbb Z}

\newcommand{\D}{\mathbb D}

\newcommand{\HH}{\mathbb H}
\newcommand{\Chat}{\widehat{\mathbb C}}
\newcommand{\id}{\operatorname{id}}
\newcommand{\Int}{\operatorname{int}}

\newcommand{\dist}{\operatorname{dist}}

\newcommand{\Var}{\operatorname{Var}}

\newcommand{\im}{\operatorname{Im}}
\newcommand{\re}{\operatorname{Re}}

\newcommand{\Mf}{\mathcal M_f}

\newcommand{\norm}[1]{\left\lVert #1\right\rVert}

\newcommand{\dd}{\,d}
\newcommand{\compactin}{\Subset}
\newcommand{\PM}{P\'erez-Marco}
\title{Unique ergodicity of hedgehogs}


\author{Ricardo P\'erez-Marco}
\address{CNRS, Institut de Math\'ematiques de Jussieu--Paris Rive Gauche,
Universit\'e Paris Cit\'e, Paris, France}

\date{22 September 2026}

\begin{document}
\begin{abstract}
We prove that the dynamics on a non-linearizable hedgehog is uniquely
ergodic. This solves a conjecture of the author formulated in 1995.
Its unique invariant probability measure is the Dirac mass at
the indifferent fixed point. The proof
relies on the techniques of quasi-invariant curves and the
hyperbolic form of the Denjoy-Yoccoz lemma developed by the author.
\end{abstract}
\maketitle
\tableofcontents
\newpage

\section{Introduction}\label{sec:statement}

Let
\[
 f(z)=e^{2\pi i\theta}z+O(z^2),\qquad \theta\in\R\setminus\mathbb Q,
\]
be a germ of holomorphic diffeomorphism at $0$. The author proved in
\cite{PM97} the existence of nontrivial
full compact connected sets $K\subset\C$ containing $0$ in any neighborhood of $0$.
We assume that
$f$ is holomorphic and one-to-one on an open neighborhood $U$ of $K$, with
\begin{equation}\label{eq:setting}
 f(0)=0,\qquad f'(0)=e^{2\pi i\theta},\qquad
 f(K)=K=f^{-1}(K).
\end{equation}
Here $f^{-1}$ is the inverse branch on $f(U)$. Recall that $K$ is
\emph{full} when $\C\setminus K$ is connected. We work with Siegel
compacta for which the external-circle correspondence of
Theorem~\ref{input:external} holds. In particular, this includes the
compacta constructed in \cite{PM97}. In the non-linearizable case they
are called \emph{hedgehogs}. They are topologically and metrically complex
invariant compacta (see \cite{PM97}, \cite{Biswas05} and \cite{Biswas16} for example). For a survey of the state of the art of the dynamics on hedgehogs we refer to \cite{Biswas09}.

Both $f$ and its inverse are defined in a Jordan neighborhood of the compactum.

\begin{definition}
A Borel probability measure $\mu$ on $K$ is $f$-invariant if
$f_*\mu=\mu$. Write $\Mf(K)$ for the set of such measures. The map
$f|_K$ is \emph{uniquely ergodic} if $\Mf(K)$ contains exactly one
measure.
\end{definition}

\begin{theorem}[Unique ergodicity of hedgehogs]\label{thm:main}
Under the preceding assumptions, if the germ of $f$ at $0$ is not
holomorphically linearizable, then
\begin{equation}\label{eq:unique}
 \Mf(K)=\{\delta_0\}.
\end{equation}
\end{theorem}

When we have unique ergodicity the Birkhoff sums converge uniformly.

\begin{corollary}
For every $\varphi\in C(K)$,
\begin{equation}\label{eq:uniform-main}
 \lim_{N\to\infty}\sup_{z\in K}
 \left|\frac1N\sum_{j=0}^{N-1}\varphi(f^j(z))-\varphi(0)\right|=0.
\end{equation}
\end{corollary}

Also the unique limit measures of atomic measures along orbits in $K$ is $\delta_0$.

\begin{corollary}
For every $r>0$,
\begin{equation}\label{eq:excursion-main}
 \lim_{N\to\infty}\sup_{z\in K}\frac1N
 \#\{0\le j<N:|f^j(z)|\ge r\}=0.
\end{equation}
\end{corollary}

The non-linearizability assumption is necessary. For the rotation
$f(z)=e^{2\pi i\theta}z$ on $K=\overline\D_R$, angular measure on every
circle $|z|=r$, $0<r\le R$, is invariant, as is $\delta_0$.

\bigskip
\textbf{Background on the problem.}

\bigskip

The conjecture was originally formulated by the author at a
Workshop on Dynamical Systems in Bologna, Italy, in 1995.

\medskip

A first result on the unique ergodicity of hedghogs was obtained
in the study of the class
of germs of quadratic type, for which a full model of the hedgehog was constructed \cite{PM95bis}.
This class of holomorphic germs does not cover  the general case.
Avila and Cheraghi \cite{AC} proved the
unique ergodicity for hedgehogs in the class of high type quadratic polynomials.
This is a very special class of quadratic polynomials, very well approximated by parabolic polynomials. This forces a heavy restriction on the rotation numbers (for example they are of zero measure). These two approaches rely  on cumbersome renormalization machinery.

\bigskip

In our present article, no arithmetic restriction is imposed and the result is valid in general for any holomorphic germ without any restriction on the rotation number $\theta$ .

\medskip

The main technical ingredients of the proof are the use of osculating quasi-invariant curves and the
hyperbolic Denjoy-Yoccoz lemma proved by the author in
\cite{PMQ,PMBB}. Quasi-invariant curves have been exploited to prove important properties
of the dynamics on hedgehogs and to resolve long standing problems which include
questions by Fatou, Dulac, Briot and Bouquet \cite{PM95,PMBB}. Initially quasi-invariant
curves were constructed using heavy renormalization machinery, developed
originally in \cite{PM96}, and later simplified in \cite{PMQ}.

\medskip

In this article we use a small height version of the hyperbolic Denjoy-Yoccoz lemma, with constants depending
on the external circle map. The idea is that Birkhoff sums are almost constant on the
quasi-invariant curves.
The Birkhoff sums associated to a continuous function $\varphi$ on $K$ are evaluated
by approximating $\varphi$ by polynomials using Mergelyan's Theorem.
The maximum principle controls the Birkhoff sums inside the
hedgehog and determines their limit from the fixed point value. The technique of using the
maximum principle to infer results on the topologically complex hedgehog is the fundamental technique used since 1995 \cite{PM95}. It seems to be the only procedure to obtain the rigidity of the dynamics on hedgehogs, without touching to the complex structure of hedgehogs.

We first prove convergence for polynomial functions, without using empty
interior. The same curves give uniform recurrence. These two facts imply
that a non-linearizable compactum has empty interior, and polynomial
approximation completes the argument. Besides elementary complex
analysis, we use the external-circle construction \cite{PM97}, Denjoy's
topological conjugacy theorem \cite{Denjoy}, and the
Mergelyan approximation theorem \cite{Mergelyan51}. The required
real and complex estimates are proved below. In particular, the bounded
return-derivative estimate of Section~\ref{sec:dk} is sufficient; its
convergence to $1$ is not needed.

Most of the techniques come from \cite{PMQ,PMBB} but we have reproved the results
here from scratch for completeness,
to keep direct track of the constant dependency and to make the article self-contained.

\bigskip

\textbf{Disclaimer.} A final step of the proof was found playing around with
ChatGPT Astra exloring to  estimate Birkhoff sums using quasi-invariant curves.
A suggestion generated in this interaction was to use Mergelyan polynomial
approximation to reduce the evaluation of Birkhoff sums to the polynomial case.
The remaining techniques and ideas are due
to the author. AI was also used for  LaTeX assistance and the TikZ figures.
The author is solely responsible for the mathematical content and correctness of the article.

\bigskip

\newpage

\section{Holomorphic Birkhoff sums}

For a positive integer $N$ and a continuous map $\varphi$, put
\begin{equation}\label{eq:birkhoff-def}
 A_N\varphi(z)=\frac1N\sum_{j=0}^{N-1}\varphi(f^j(z)).
\end{equation}
We begin with the maximum-principle argument that will be used at the
end of the construction.

\begin{proposition}[Boundary matching criterion]\label{prop:criterion}
Suppose that $\Omega_n$ are bounded Jordan domains containing $K$, with
boundaries $\Gamma_n$, and that $Q_n\to\infty$. Assume:
\begin{enumerate}[label=\textnormal{(\alph*)}]
\item every $f^j$, $0\le j<Q_n$, is holomorphic on a neighborhood of
      $\overline{\Omega_n}$;
\item all points $f^j(\zeta)$ with $\zeta\in\Gamma_n$ and $0\le j<Q_n$
      lie in one fixed disk;
\item there are $\varepsilon_n\to0$ such that, for any
      $\zeta,\xi\in\Gamma_n$, some permutation $\sigma$ of
      $\{0,\ldots,Q_n-1\}$ satisfies
      \begin{equation}\label{eq:abstract-matching}
       \max_{0\le j<Q_n}|f^j(\zeta)-f^{\sigma(j)}(\xi)|
       \le\varepsilon_n.
      \end{equation}
\end{enumerate}
Then for every polynomial $P$,
\begin{equation}\label{eq:poly-criterion}
 \norm{A_{Q_n}P-P(0)}_K\longrightarrow0.
\end{equation}
\end{proposition}

\begin{proof}
Choose a disk containing the points in (b), and set
$L_P=\max|P'|$ on that disk. Its convexity gives
$|P(u)-P(v)|\le L_P|u-v|$ there. For $\zeta,\xi\in\Gamma_n$, use the
permutation from (c). Since a permutation preserves a finite sum,
\begin{align*}
 |A_{Q_n}P(\zeta)-A_{Q_n}P(\xi)|
 &\le\frac1{Q_n}\sum_{j=0}^{Q_n-1}
 |P(f^j(\zeta))-P(f^{\sigma(j)}(\xi))|\\
 &\le L_P\varepsilon_n.
\end{align*}
Fix $\zeta_n\in\Gamma_n$ and write $b_n=A_{Q_n}P(\zeta_n)$. The maximum
principle gives
\[
 \norm{A_{Q_n}P-b_n}_{\overline{\Omega_n}}\le L_P\varepsilon_n.
\]
Because $0\in\Omega_n$ and $f(0)=0$, we have $A_{Q_n}P(0)=P(0)$.
Thus $|b_n-P(0)|\le L_P\varepsilon_n$, and
\begin{equation}\label{eq:poly-two}
 \norm{A_{Q_n}P-P(0)}_K\le2L_P\varepsilon_n.
\end{equation}
\end{proof}

\section{External coordinates}\label{sec:external}

\begin{hypothesis}[External-circle correspondence, \PM\ \cite{PM97}]\label{input:external}
For the Siegel compactum $K$, let
\[
 h:\Chat\setminus\overline\D\longrightarrow\Chat\setminus K,
 \qquad h(\infty)=\infty,
\]
be the normalized exterior Riemann map. The map induced by $f$ near the
circle in these coordinates extends analytically across the circle and
restricts to an orientation-preserving analytic circle diffeomorphism.
Its rotation number is irrational; with a consistently oriented
parametrization it agrees with the fixed-point rotation number.
\end{hypothesis}
See \cite[Theorem 2 and the fundamental construction]{PM97} and
\cite[Section 1]{PMQ}. The conformal representation $h$
does not have a continuous
extension to the unit circle because $K$ is not locally connected.

Use upper-half-plane coordinates
\begin{equation}\label{eq:psi}
 \Psi(w)=h(e^{-2\pi i w}),\qquad w\in\HH:=\{\im w>0\}.
\end{equation}
The map $\Psi$ is $1$-periodic and induces a conformal bijection from
$\HH/\Z$ onto $\C\setminus K$. The upper half-plane correspond to the exterior disk.

The external map has an increasing real lift $g$ with
\begin{equation}\label{eq:g-lift}
 g(x+1)=g(x)+1,\qquad g'(x)>0,
 \qquad f\circ\Psi=\Psi\circ g
\end{equation}
in a sufficiently thin upper strip. Changing the lift by an integer,
we may assume its rotation number is $\beta\in(0,1)\setminus\mathbb Q$.
Its sign relative to $\theta$ depends on the circle orientation and has
no role in the proof.

Shrink the strip so that, for some $\Delta>0$, $g$ is holomorphic on a
neighborhood of
\[
 S_\Delta=\{w:|\im w|<\Delta\},
\]
its derivative has no zeros there, and
\begin{equation}\label{eq:tau}
 \tau:=\sup_{S_\Delta}\left|\frac{g''}{g'}\right|<\infty.
\end{equation}
These properties follow from $g'>0$ on $\R$ and compactness modulo
$\Z$, after restricting a slightly larger strip. We choose the analytic
logarithm of $g'$ that is real on $\R$.

\begin{hypothesis}[Denjoy \cite{Denjoy}]\label{input:denjoy}
An orientation-preserving $C^1$ circle diffeomorphism with irrational
rotation number and $\log g'$ of bounded variation is topologically
conjugate to the corresponding rigid rotation.
\end{hypothesis}
Our analytic map satisfies these hypotheses. Hence there is an
increasing homeomorphism $H:\R\to\R$ such that
\begin{equation}\label{eq:H}
 H(t+1)=H(t)+1,\qquad g(H(t))=H(t+\beta).
\end{equation}
Only the monotonicity and continuity of $H$ will be used. No
differentiability or absolute continuity of this conjugacy is assumed.

\section{Continued fractions and orbit matching}\label{sec:cf}

Write $\beta=[0;a_1,a_2,\ldots]$. Set
\[
 p_{-1}=1,\quad p_0=0,\qquad q_{-1}=0,\quad q_0=1,
\]
and, for $n\ge0$,
\[
 p_{n+1}=a_{n+1}p_n+p_{n-1},\qquad
 q_{n+1}=a_{n+1}q_n+q_{n-1}.
\]
Denote
\begin{equation}\label{eq:cf-errors}
 e_n=q_n\beta-p_n,\qquad \delta_n=|e_n|,
 \qquad Q_n=q_{n+1}.
\end{equation}
We disregard finitely many initial indices whenever necessary.

\begin{lemma}[Approximation and separation]\label{lem:cf}
The denominators tend to infinity; consecutive denominators are coprime;
$e_n$ and $e_{n+1}$ have opposite signs; and
\begin{equation}\label{eq:cf-bounds}
 \frac1{q_{n+1}+q_n}<\delta_n<\frac1{q_{n+1}}.
\end{equation}
If $0<r<q_{n+1}$ and $s\in\Z$, then
\begin{equation}\label{eq:best-approx}
 |r\beta-s|\ge\delta_n.
\end{equation}
Consequently, for any $t\in\R$, the oriented circle intervals with
endpoints
\[
 t+j\beta,\qquad t+j\beta+e_n,
 \qquad 0\le j<q_{n+1},
\]
have pairwise disjoint interiors.
\end{lemma}

\begin{proof}
The recurrences give
$|p_nq_{n+1}-p_{n+1}q_n|=1$, with alternating sign. In particular, both
$p_n,q_n$ and $q_n,q_{n+1}$ are relatively prime. Growth of the
recurrences gives $q_n\to\infty$.

Let $b_{n+1}=[a_{n+1};a_{n+2},\ldots]$. The finite continued-fraction
identity, obtained by induction on the defining fractional-linear
transformations, is
\[
 \beta=\frac{p_n b_{n+1}+p_{n-1}}
              {q_n b_{n+1}+q_{n-1}}.
\]
Taking the difference from $p_n/q_n$ gives alternating error signs and
\[
 \delta_n=\frac1{q_n b_{n+1}+q_{n-1}}
          =\frac1{q_{n+1}+q_n/b_{n+2}}.
\]
Since $b_{n+2}>1$, this proves \eqref{eq:cf-bounds}.

For \eqref{eq:best-approx}, use the unimodular basis of integer vectors
$(p_n,q_n),(p_{n+1},q_{n+1})$ to write
\[
 (s,r)=a(p_n,q_n)+b(p_{n+1},q_{n+1}),\qquad a,b\in\Z.
\]
If $b\ge1$, the inequality $r<q_{n+1}$ forces $a\le-1$.
If $b\le-1$, the inequality $r>0$ forces $a\ge1$.
In either case the two terms in
$r\beta-s=ae_n+be_{n+1}$ have the same sign, and its absolute value is
at least $\delta_n$. If $b=0$, then $a\ge1$ and the conclusion is
immediate.

For large $n$, $\delta_n<1/2$. Interior intersection of two oriented
intervals of length $\delta_n$ would imply that the circle distance
between their starting points is less than $\delta_n$. Their index
difference has absolute value strictly between $0$ and $q_{n+1}$, in
contradiction with \eqref{eq:best-approx}.
\end{proof}

\begin{lemma}[Permutation matching of two rotation segments]\label{lem:rotation-match}
For every $t,s\in\R$ there is a permutation $\sigma$ of
$\{0,\ldots,Q_n-1\}$ satisfying
\begin{equation}\label{eq:phase-match}
 \dist_{\R/\Z}\bigl(t+j\beta,s+\sigma(j)\beta\bigr)<3\delta_n
 \qquad(0\le j<Q_n).
\end{equation}
\end{lemma}

\begin{proof}
Write $Q=q_{n+1}$ and $P=p_{n+1}$. Choose an integer $k$ with
\[
 |t-s-k/Q|\le\frac1{2Q}.
\]
Because $P$ is invertible modulo $Q$, the congruence
\begin{equation}\label{eq:permutation-explicit}
 \sigma(j)P\equiv jP+k\pmod Q
\end{equation}
defines a permutation. Modulo an integer,
\begin{align*}
 t+j\beta-s-\sigma(j)\beta
 &=t-s-k/Q+(j-\sigma(j))(\beta-P/Q).
\end{align*}
Its circle distance from zero is at most
\[
 \frac1{2Q}+Q|\beta-P/Q|
 =\frac1{2Q}+\delta_{n+1}<\frac3{2Q}.
\]
On the other hand, \eqref{eq:cf-bounds} gives
$\delta_n>1/(Q+q_n)>1/(2Q)$. This proves the claim.
\end{proof}

The matching does not need to preserve circular  ordering.

\section{A real return estimate}\label{sec:dk}

Set
\begin{equation}\label{eq:real-constants}
 V=\Var_{\R/\Z}(\log g'),\qquad B=e^V,\qquad \Lambda=e^{4V}.
\end{equation}
These are fixed finite constants depending on $g$.

\begin{lemma}[A coarse Denjoy--Koksma estimate]\label{lem:dk}
If $\phi:\R/\Z\to\R$ is continuous and of bounded variation, then
\begin{equation}\label{eq:dk}
 \left|\sum_{j=0}^{q_n-1}\phi(t+j\beta)
      -q_n\int_0^1\phi(u)\dd u\right|
 \le4\Var(\phi)
\end{equation}
for every $t$ and every sufficiently large $n$.
\end{lemma}

\begin{proof}
Put $q=q_n$ and $p=p_n$. The points $t+jp/q$, $0\le j<q$, form a
uniform $q$-point grid on the circle. Each corresponding point
$t+j\beta$ has circle distance less than
\[
 q|\beta-p/q|=\delta_n<1/q
\]
from its grid point. A circle interval $J$ of length $|J|$ contains
between $q|J|-2$ and $q|J|+2$ points of the grid, allowing either endpoint
convention. Enlarging $J$ by $\delta_n$ at both ends, or deleting such
endpoint pieces, shows
\begin{equation}\label{eq:count-discrepancy}
 \left|\#\{0\le j<q:t+j\beta\in J\}-q|J|\right|\le4.
\end{equation}
If an enlarged interval is the whole circle, or a deleted interval is
empty, the same bound follows directly.

Let
\[
 \nu=\sum_{j=0}^{q-1}\delta_{t+j\beta}-q\,du.
\]
It has total mass zero. By \eqref{eq:count-discrepancy}, its cumulative
mass on any interval starting at $0$ has absolute value at most $4$.
For a continuously differentiable periodic $\phi$, integration by parts,
or simply Fubini applied to
$\phi(x)=\phi(1)-\int_x^1\phi'(u)\dd u$, gives
\[
 \left|\int\phi\dd\nu\right|\le4\int_0^1|\phi'(u)|\dd u.
\]
A continuous function of bounded variation can be approximated uniformly
by its piecewise-linear interpolants, whose total variations do not
exceed $\Var(\phi)$. Applying the same integration-by-parts formula to
the interpolants and passing to the limit proves \eqref{eq:dk}.
\end{proof}

\begin{proposition}[Uniform derivative bound at denominators]\label{prop:denjoy-bound}
For all sufficiently large $n$ and every $x\in\R$,
\begin{equation}\label{eq:denjoy-bound}
 \Lambda^{-1}\le (g^{q_n})'(x)\le\Lambda.
\end{equation}
\end{proposition}

\begin{proof}
Apply Lemma~\ref{lem:dk} to
\[
 \phi(t)=\log g'(H(t)).
\]
The function is continuous and has variation $V$, since $H$ is
increasing and has degree one. Put $I=\int_0^1\phi(t)\dd t$.
The chain rule and \eqref{eq:H} imply
\[
 \log(g^{q_n})'(H(t))=\sum_{j=0}^{q_n-1}\phi(t+j\beta).
\]
Thus, for every $x$,
\[
 q_n I-4V\le\log(g^{q_n})'(x)\le q_n I+4V.
\]
The lift $g^{q_n}$ satisfies $g^{q_n}(x+1)=g^{q_n}(x)+1$, so
\[
 \int_0^1(g^{q_n})'(x)\dd x=1.
\]
Integrating the exponential bounds gives $|q_nI|\le4V$. Since
$q_n\to\infty$, necessarily $I=0$. Substitution yields
\eqref{eq:denjoy-bound}.
\end{proof}

`
\section{Return intervals and distortion}\label{sec:return}

Define
\begin{equation}\label{eq:return-def}
 F_n(x)=g^{q_n}(x)-p_n,\qquad
 m_n(x)=F_n(x)-x,\qquad
 \ell_n(x)=|m_n(x)|,\qquad M_n=\max_x\ell_n(x).
\end{equation}
Let $I_n(x)$ be the unoriented closed interval with endpoints $x$ and
$F_n(x)$. The function $m_n$ is $1$-periodic and has constant nonzero
sign. Indeed, from \eqref{eq:H},
\begin{equation}\label{eq:return-conjugacy}
 F_n(H(t))=H(t+e_n).
\end{equation}
Monotonicity of $H$ gives the sign assertion. Its uniform continuity
modulo one and $e_n\to0$ give
\begin{equation}\label{eq:Mn}
 M_n\longrightarrow0.
\end{equation}
For each fixed $n$, $\ell_n$ is positive, periodic, and real analytic;
its positive minimum need not be bounded uniformly in $n$.

\begin{lemma}[Disjointness and total length]\label{lem:disjoint}
For every $x$, the intervals $g^j(I_n(x))$, $0\le j<Q_n$, have
pairwise disjoint interiors modulo $\Z$, and
\begin{equation}\label{eq:length-sum}
 \sum_{j=0}^{Q_n-1}\ell_n(g^j(x))\le1.
\end{equation}
\end{lemma}

\begin{proof}
By \eqref{eq:return-conjugacy}, these intervals are the images under
$H$ of the disjoint rotation intervals in Lemma~\ref{lem:cf}.
The maps $F_n$ and $g^j$ commute, because $g$ commutes with integer
translations. Consequently
\[
 g^j(I_n(x))=I_n(g^j(x)).
\]
Their lengths are therefore $\ell_n(g^j(x))$. Disjointness on a circle
of length one gives \eqref{eq:length-sum}.
\end{proof}

\begin{lemma}[Derivative versus return-length ratio]\label{lem:distortion}
For $0\le j\le Q_n$ and all $x$,
\begin{equation}\label{eq:distortion}
 B^{-1}\ell_n(g^j(x))
 \le (g^j)'(x)\ell_n(x)
 \le B\ell_n(g^j(x)).
\end{equation}
\end{lemma}

\begin{proof}
If $u,v\in I_n(x)$, then
\begin{align*}
 |\log(g^j)'(u)-\log(g^j)'(v)|
 &\le\sum_{l=0}^{j-1}
 |\log g'(g^l(u))-\log g'(g^l(v))|\\
 &\le\sum_{l=0}^{j-1}\Var_{g^l(I_n(x))}(\log g')\le V.
\end{align*}
The last inequality uses Lemma~\ref{lem:disjoint}; it remains valid at
$j=Q_n$. By the mean-value theorem, some $v\in I_n(x)$ satisfies
\[
 (g^j)'(v)=\frac{|g^j(I_n(x))|}{|I_n(x)|}
         =\frac{\ell_n(g^j(x))}{\ell_n(x)}.
\]
Exponentiating the preceding oscillation bound proves the result.
\end{proof}

\begin{lemma}[Variation of the return height on one interval]\label{lem:height-real}
For sufficiently large $n$,
\begin{equation}\label{eq:ell-slope}
 |\ell_n'(x)|\le\Lambda-1
\end{equation}
for all $x$. If $u\in I_n(x)$, then
\begin{equation}\label{eq:ell-comparable}
 \Lambda^{-1}\ell_n(x)\le\ell_n(u)\le\Lambda\ell_n(x).
\end{equation}
\end{lemma}

\begin{proof}
Proposition~\ref{prop:denjoy-bound} gives
$\Lambda^{-1}\le F_n'\le\Lambda$. Since
$\ell_n'=\operatorname{sign}(m_n)(F_n'-1)$, this implies
\eqref{eq:ell-slope}.

If $F_n(x)>x$ and $u=x+d$, $0\le d\le\ell_n(x)$, then
\[
 \ell_n(u)=\ell_n(x)+\int_x^{x+d}(F_n'(v)-1)\dd v.
\]
Its lower and upper bounds are
$\ell_n(x)+(\Lambda^{-1}-1)d$ and
$\ell_n(x)+(\Lambda-1)d$, respectively. They lie between
$\Lambda^{-1}\ell_n(x)$ and $\Lambda\ell_n(x)$.

If $F_n(x)<x$ and $u=x-d$, $0\le d\le\ell_n(x)$, then instead
\[
 \ell_n(u)=\ell_n(x)+\int_{x-d}^{x}(F_n'(v)-1)\dd v,
\]
and the same bounds apply. Thus the same comparison holds for both signs of the closest return.
\end{proof}

\section{The hyperbolic Denjoy-Yoccoz lemma}\label{sec:tube}

We use the hyperbolic form of the Denjoy-Yoccoz lemma due to the
author \cite[Sections 3-4]{PMQ}. We give the small-height version
needed here, with a bound depending on $g$.

On $\HH$ we use the curvature $-1$ hyperbolic metric
\begin{equation}\label{eq:hyperbolic}
 ds_{\HH}=\frac{|dw|}{\im w},
\end{equation}
and denote its distance by $d_{\HH}$. Choose a fixed number $c>0$ such
that
\begin{equation}\label{eq:small-c}
 c\le1,\qquad 2\tau cB\le\log(5/4).
\end{equation}
The map $g$ is fixed before $c$ is chosen. Thus no smallness assumption
on $\tau$ is required. Take $n$ sufficiently large that
\begin{equation}\label{eq:strip-height}
 t_n:=2cBM_n<\Delta.
\end{equation}
The osculating graph in external coordinates is
\begin{equation}\label{eq:gamma}
 \gamma_n(x)=x+i c\ell_n(x),\qquad x\in\R.
\end{equation}

\begin{figure}[tbp]
\centering
\begin{tikzpicture}[x=1cm,y=1cm,>=Latex,font=\small,
  declare function={qh(\x)=1.3+0.2*sin(55*\x)+0.1*cos(110*\x);}]
  \draw[->] (-.2,0) -- (13.6,0) node[below left] {$\re w$};
  \draw[->] (0,-.15) -- (0,4.0) node[above] {$\im w$};
  \coordinate (A) at (1.65,{qh(1.65)});
  \coordinate (B) at (8.30,{qh(8.30)});
  \coordinate (Z) at (8.82,{qh(8.30)+.50});
  \pgfmathsetmacro{\yb}{qh(8.30)}
  \pgfmathsetmacro{\bc}{\yb*cosh(.62)}
  \pgfmathsetmacro{\br}{\yb*sinh(.62)}
  \draw[gray!65,fill=gray!7] (8.30,\bc) circle[radius=\br];
  \draw[curveblue,thick,smooth,samples=180,domain=.20:13.1]
    plot(\x,{qh(\x)});
  \node[curveblue,anchor=west] at (10.30,1.70)
    {$\gamma_n:\ y=c\ell_n(x)$};
  \draw[densely dashed,gray] (1.65,0) -- (A);
  \draw[densely dashed,gray] (8.30,0) -- (B);
  \node[below] at (1.65,0) {$x$};
  \node[below] at (8.30,0) {$x_j=g^j(x)$};
  \draw[orbitred,thick,->] (A) .. controls (3.45,3.75) and (6.90,3.75) .. (Z);
  \node[orbitred] at (5.3,3.63) {$g^j,\quad 0\le j\le Q_n$};
  \fill[curveblue] (A) circle (1.7pt);
  \fill[curveblue] (B) circle (1.7pt);
  \fill[orbitred] (Z) circle (1.7pt);
  \node[curveblue,anchor=east] at (3.6,2.10) {$\gamma_n(x)$};
  \draw[curveblue,thin] (2.30,1.95) -- (A);
  \node[curveblue,anchor=east] at (7.55,.70) {$\gamma_n(x_j)$};
  \draw[curveblue,thin] (7.62,.74) -- (B);
  \node[orbitred,anchor=west] at (10.0,3.03) {$g^j(\gamma_n(x))$};
  \draw[orbitred,thin] (9.95,2.87) -- (Z);
  \draw[densely dashed,thin] (B) -- (Z);
  \node[anchor=west] at (10.02,2.48) {$d_{\HH}\le D$};
\end{tikzpicture}
\caption{The hyperbolic Denjoy--Yoccoz estimate. The actual iterate
$g^j(\gamma_n(x))$ lies in a hyperbolic ball of fixed radius $D$ about
$\gamma_n(g^j(x))$. The graph height is $c\ell_n(x)$; the drawing is
schematic.}
\label{fig:tube}
\end{figure}
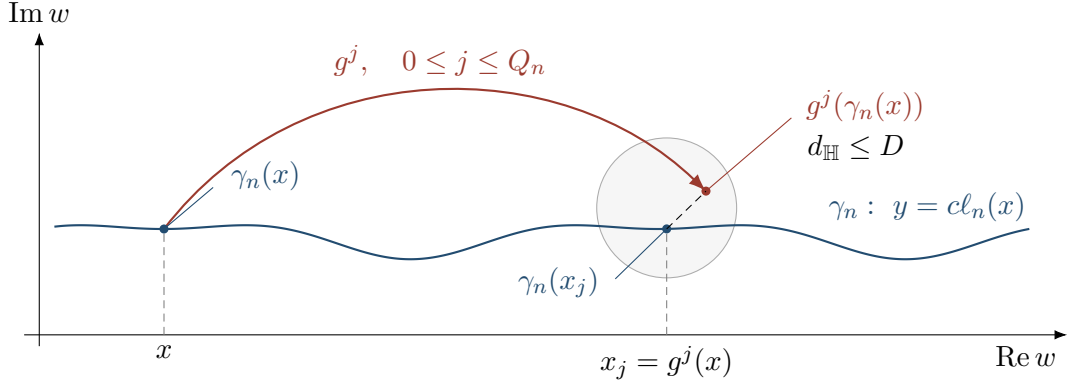

\begin{proposition}[Small-height Denjoy--Yoccoz lemma]\label{prop:tube}
For $0\le s\le c$, set $z_s=x+i s\ell_n(x)$ and $x_j=g^j(x)$.
All iterates $g^j(z_s)$, $0\le j\le Q_n$, are defined. For every such
$j$ and $s>0$,
\begin{equation}\label{eq:vertical-expansion}
 g^j(z_s)=x_j+i s(g^j)'(x)\ell_n(x)(1+E_{j,s}),
 \qquad |E_{j,s}|\le\tfrac14.
\end{equation}
In particular,
\begin{align}
 |g^j(z_s)-x_j|&\le2sB\ell_n(x_j),\label{eq:tube-bootstrap}\\
 \frac{3s}{4B}\ell_n(x_j)
 \le\im g^j(z_s)&\le\frac{5sB}{4}\ell_n(x_j)<t_n.
 \label{eq:tube-im}
\end{align}
There is a constant $D<\infty$, independent of $n,x,j$, such that
\begin{equation}\label{eq:tube-hyperbolic}
 d_{\HH}\bigl(g^j(\gamma_n(x)),\gamma_n(g^j(x))\bigr)\le D
 \qquad(0\le j\le Q_n).
\end{equation}
\end{proposition}

\begin{proof}
We prove \eqref{eq:tube-bootstrap} by induction on $j$, simultaneously
for all $s\in[0,c]$. It is immediate when $j=0$. Suppose it holds for
all earlier iterates. Their distances from their real base points are
at most $2cBM_n<\Delta$, so the next composition is defined.
The finite list of preceding compositions is holomorphic on a
neighborhood of each initial vertical segment.

Using the analytic logarithm of $g'$ and the chain rule, we have
\begin{align}
 L_j(s)&:=\sum_{l=0}^{j-1}
   \bigl(\log g'(g^l(z_s))-\log g'(x_l)\bigr),\notag\\
 e^{L_j(s)}&=\frac{(g^j)'(z_s)}{(g^j)'(x)}.
 \label{eq:log-ratio}
\end{align}
The straight segment joining $x_l$ to $g^l(z_s)$ lies in the strip.
The definition of $\tau$, the induction hypothesis, and
\eqref{eq:length-sum} therefore give
\begin{align}
 |L_j(s)|
 &\le\tau\sum_{l=0}^{j-1}|g^l(z_s)-x_l|\notag\\
 &\le2\tau sB\sum_{l=0}^{j-1}\ell_n(x_l)
 \le2\tau sB\le\log(5/4).
 \label{eq:log-bootstrap}
\end{align}
As $|e^w-1|\le e^{|w|}-1$, it follows that
\begin{equation}\label{eq:derivative-quarter}
 \left|\frac{(g^j)'(z_s)}{(g^j)'(x)}-1\right|\le\tfrac14.
\end{equation}
This estimate holds at every point of the initial vertical segment.
Integrating its derivative along that segment proves
\eqref{eq:vertical-expansion}. Lemma~\ref{lem:distortion} then gives
\[
 |g^j(z_s)-x_j|
 \le\tfrac54s(g^j)'(x)\ell_n(x)
 \le\tfrac54sB\ell_n(x_j),
\]
which is stronger than the induction claim. The induction closes up to
$j=Q_n$. Since the real derivative is positive, taking imaginary parts
in \eqref{eq:vertical-expansion} proves \eqref{eq:tube-im}.

For $s=c$, normalize by the real affine hyperbolic isometry
\[
 w\longmapsto\frac{w-x_j}{c\ell_n(x_j)}.
\]
The model point $\gamma_n(x_j)$ becomes $i$. The image $v$ of
$g^j(\gamma_n(x))$ satisfies
\[
 |\re v|\le B/4,\qquad
 \frac3{4B}\le\im v\le\frac{5B}{4}.
\]
A horizontal segment followed by a vertical segment has hyperbolic
length at most
\[
 \frac{|\re v|}{\im v}+|\log\im v|
 \le \frac{B^2}{3}+\log(4B/3).
\]
Thus one may take
\begin{equation}\label{eq:D}
 D=\frac{B^2}{3}+\log(4B/3).
\end{equation}
\end{proof}

\begin{remark}
The constant $D$ depends on $g$, but not on $n$, $x$, or $j$. We need a
bounded hyperbolic error, not an arbitrarily small one. This is why the
small-height form above suffices.
\end{remark}

\section{Quasi-invariant curves and orbit matching}\label{sec:graph-match}

The quotient graph $\gamma_n(\R)/\Z$ is an analytic Jordan curve in the
cylinder. Let $d_{\HH/\Z}$ denote the quotient distance
\[
 d_{\HH/\Z}(u,v)=\inf_{k\in\Z}d_{\HH}(u,v+k).
\]

\begin{lemma}[Bounded graph arcs]\label{lem:graph-arcs}
The graph arc of $\gamma_n$ over any interval $I_n(x)$ has hyperbolic
length at most
\begin{equation}\label{eq:L}
 L=\frac{\Lambda}{c}\sqrt{1+c^2(\Lambda-1)^2}.
\end{equation}
If $\dist_{\R/\Z}(t,s)<3\delta_n$, then
\begin{equation}\label{eq:phase-graph}
 d_{\HH/\Z}\bigl(\gamma_n(H(t)),\gamma_n(H(s))\bigr)\le4L.
\end{equation}
\end{lemma}

\begin{proof}
By Lemma~\ref{lem:height-real}, the length of the first arc is at most
\begin{align*}
 \int_{I_n(x)}
 \frac{\sqrt{1+c^2\ell_n'(u)^2}}{c\ell_n(u)}\dd u
 &\le
 \frac{\sqrt{1+c^2(\Lambda-1)^2}}
      {c\Lambda^{-1}\ell_n(x)}\,|I_n(x)|=L.
\end{align*}

Choose representatives of $t,s$ with $|t-s|<3\delta_n$.
In phase coordinates, a return interval is an interval of length
$\delta_n$ with signed displacement $e_n$. Starting at $t$, the segment
to $s$ is covered by at most four consecutive such intervals, with the
last possibly shortened. This works for either sign of $e_n$; traversing
an interval backwards does not change its length. Equation
\eqref{eq:return-conjugacy} transports these intervals by $H$ to
consecutive intervals of the form $I_n(u)$. Their graph arcs give the
claimed path of length at most $4L$.
\end{proof}

\begin{proposition}[Quasi-invariance and closest returns]\label{prop:quasi}
For $0\le j\le Q_n$, the hyperbolic Hausdorff distance in the cylinder
between $g^j(\gamma_n)$ and $\gamma_n$ is at most $D$. Moreover,
\begin{equation}\label{eq:closest-hyperbolic}
 d_{\HH/\Z}\bigl(g^{q_n}(\gamma_n(x)),\gamma_n(x)\bigr)\le D+L
\end{equation}
for every $x$.
\end{proposition}

\begin{proof}
The tube estimate pairs $g^j(\gamma_n(x))$ with
$\gamma_n(g^j(x))$ within distance $D$. Since $g^j$ maps $\R$ onto
$\R$, this also proves the reverse directed Hausdorff bound.
For the last assertion, note that
\[
 \gamma_n(g^{q_n}(x))=\gamma_n(F_n(x))+p_n.
\]
The graph points over $x$ and $F_n(x)$ are joined by a single arc from
Lemma~\ref{lem:graph-arcs}. Combine that arc with the tube bound.
\end{proof}

\begin{figure}[tbp]
\centering
\begin{tikzpicture}[x=1cm,y=1cm,>=Latex,font=\small]
  \node[font=\small\bfseries] at (3.15,4.05) {Rotation coordinates};
  \draw[->,gray] (.35,2.85) -- (6.0,2.85);
  \draw[->,gray] (.35,1.48) -- (6.0,1.48);
  \node[anchor=east] at (.37,2.85) {$t$};
  \node[anchor=east] at (.37,1.48) {$s$};
  \foreach \a/\b in {.85/1.12,1.87/2.03,2.83/3.10,3.94/4.12,5.03/5.20} {
    \draw[gray!65,->] (\a,2.77) -- (\b,1.57);
    \fill[curveblue] (\a,2.85) circle (1.6pt);
    \fill[orbitred] (\b,1.48) circle (1.6pt);
  }
  \node[curveblue] at (3.1,3.33) {$t+j\beta$};
  \node[orbitred] at (3.1,.99) {$s+\sigma(j)\beta$};
  \node[align=center,font=\footnotesize] at (3.2,.22)
    {$\dist_{\R/\Z}(t+j\beta,s+\sigma(j)\beta)<3\delta_n$};
  \node[font=\small\bfseries] at (10.00,4.05) {External coordinates};
  \coordinate (M1) at (7.25,1.10);
  \coordinate (M2) at (12.35,1.10);
  \coordinate (Z1) at (7.68,2.95);
  \coordinate (Z2) at (11.89,2.95);
  \draw[curveblue,thick] (6.95,1.16)
     .. controls (7.05,1.13) and (7.15,1.11) .. (M1)
     .. controls (8.20,.75) and (8.30,1.36) .. (9.70,1.17)
     .. controls (10.80,.97) and (11.55,.90) .. (M2)
     .. controls (12.48,1.12) and (12.60,1.15) .. (12.70,1.18);
  \draw[gray!75] (M1) -- (Z1) node[midway,left,text=black] {$\le D$};
  \draw[gray!75] (M2) -- (Z2) node[midway,right,text=black] {$\le D$};
  \draw[densely dashed,thick] (Z1) -- (Z2)
     node[midway,above=3pt] {$\le 2D+4L$};
  \draw[decorate,decoration={brace,mirror,amplitude=4pt}]
     (7.35,.67) -- (12.25,.67) node[midway,below=5pt] {$\le 4L$};
  \fill[curveblue] (M1) circle (1.7pt);
  \fill[curveblue] (M2) circle (1.7pt);
  \fill[orbitred] (Z1) circle (1.7pt);
  \fill[orbitred] (Z2) circle (1.7pt);
  \node[orbitred,anchor=east] at (7.58,3.36) {$z_j$};
  \node[orbitred,anchor=west] at (12.00,3.36) {$w_{\sigma(j)}$};
  \node[curveblue,anchor=east] at (7.00,.99) {$a_j$};
  \node[curveblue,anchor=west] at (12.62,.98) {$b_{\sigma(j)}$};
\end{tikzpicture}
\caption{Permutation matching. The left diagram represents the matching
of the two phase sets, displayed in their circular order. On the right,
$a_j=\gamma_n(H(t+j\beta))$ and
$b_{\sigma(j)}=\gamma_n(H(s+\sigma(j)\beta))$ are model points;
$z_j=g^j(\gamma_n(H(t)))$ and
$w_{\sigma(j)}=g^{\sigma(j)}(\gamma_n(H(s)))$ are actual iterates.
The bounds $D$, $4L$, and $D$ give $C=2D+4L$. All distances on the
right are hyperbolic distances in $\HH/\Z$. The diagrams are schematic.}
\label{fig:matching}
\end{figure}
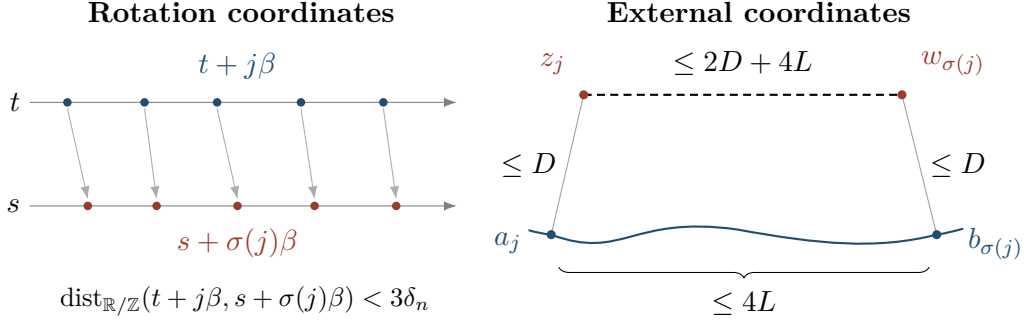

\begin{proposition}[Hyperbolic permutation matching]\label{prop:hyperbolic-match}
Set
\begin{equation}\label{eq:Cmatch}
 C=2D+4L.
\end{equation}
For every $x,y\in\R$ and every sufficiently large $n$, some permutation
$\sigma$ of $\{0,\ldots,Q_n-1\}$ satisfies
\begin{equation}\label{eq:hyperbolic-match}
 d_{\HH/\Z}\bigl(g^j(\gamma_n(x)),
                 g^{\sigma(j)}(\gamma_n(y))\bigr)\le C
 \quad(0\le j<Q_n).
\end{equation}
All the points in this formula have imaginary part at most $t_n$.
\end{proposition}

\begin{proof}
Write $x=H(t)$ and $y=H(s)$ and choose the permutation of
Lemma~\ref{lem:rotation-match}. The corresponding model points are
\[
 \gamma_n(H(t+j\beta)),\qquad
 \gamma_n(H(s+\sigma(j)\beta)).
\]
They lie within distance $4L$ by
\eqref{eq:phase-match} and \eqref{eq:phase-graph}.
Each actual orbit point lies within distance $D$ of its model point,
by Proposition~\ref{prop:tube}. The triangle inequality gives
\eqref{eq:hyperbolic-match}. The height statement follows from
\eqref{eq:tube-im}.
\end{proof}

\section{The exterior conformal map}\label{sec:boundary}

We now transfer the estimates in the hyperbolic metric
to the dynamical plane. A bounded
hyperbolic displacement near the real axis has an image of small
Euclidean diameter under $\Psi$. The following finite-area argument
gives the required uniform statement.

\begin{lemma}[Bounded and shrinking exterior collars]\label{lem:collar}
For every sufficiently small fixed $a>0$, the set
\[
 \Psi\bigl(\{0<\im w\le a\}/\Z\bigr)
\]
is bounded. Moreover,
\begin{equation}\label{eq:collar-shrink}
 \lim_{t\downarrow0}\sup_{0<\im w\le t}\dist(\Psi(w),K)=0.
\end{equation}
\end{lemma}

\begin{proof}
The image by $h$ of the circle $|\zeta|=e^{2\pi a}$ is a Jordan curve
surrounding $K$. The image of $1<|\zeta|\le e^{2\pi a}$ is in its bounded
closed interior, proving boundedness. Equivalently, boundedness follows
from the Laurent expansion of $h$ and its inverse at infinity.

If \eqref{eq:collar-shrink} failed, there would be $w_k$ with
$\im w_k\to0$ and $\Psi(w_k)$ staying a positive distance from $K$.
By boundedness, a subsequence would converge to a finite point of
$\C\setminus K$. The inverse conformal map $h^{-1}$ is continuous there,
so the moduli $|e^{-2\pi i w_k}|$ would converge to a number strictly
larger than $1$. But those moduli are $e^{2\pi\im w_k}\to1$, a
contradiction.
\end{proof}

\begin{lemma}[Vanishing derivative at hyperbolic boundary scale]\label{lem:little-bloch}
Define, for sufficiently small $t>0$,
\[
 \eta(t)=\sup_{0<\im w\le t}(\im w)|\Psi'(w)|.
\]
Then
\begin{equation}\label{eq:little-bloch}
 \eta(t)\longrightarrow0\qquad(t\downarrow0).
\end{equation}
\end{lemma}

\begin{proof}
Choose $a<1/4$. Since $\Psi$ is conformal and one-to-one modulo $\Z$,
the change-of-variables formula and Lemma~\ref{lem:collar} give
\begin{equation}\label{eq:finite-area}
 \mathcal E(a):=\int_0^a\int_0^1|\Psi'(x+iy)|^2\dd x\dd y<\infty.
\end{equation}
Indeed this integral is the area of its bounded image collar; no part
of $K$ is included in that area computation. Absolute continuity of the
Lebesgue integral gives $\mathcal E(s)\to0$ as $s\downarrow0$.

Let $w=x+iy$, with $2y<a$. The disk $B(w,y/2)$ lies in
$\{0<\im u<2y\}$. The submean inequality for the subharmonic function
$|\Psi'|^2$ gives
\[
 |\Psi'(w)|^2\le\frac4{\pi y^2}
                  \int_{B(w,y/2)}|\Psi'(u)|^2\dd A(u).
\]
By periodicity, this last integral is at most $2\mathcal E(2y)$;
the disk meets at most two adjacent fundamental strips, so the bound
remains valid when it crosses a period boundary. Hence
\begin{equation}\label{eq:little-bloch-bound}
 y^2|\Psi'(x+iy)|^2\le\frac8\pi\mathcal E(2y),\qquad
 \eta(t)\le\sqrt{\frac8\pi\mathcal E(2t)}.
\end{equation}
This proves the claim.
\end{proof}

\begin{lemma}[Bounded hyperbolic distance becomes uniformly small]\label{lem:metric-conversion}
For each fixed $A<\infty$, there is a function $\omega_A(t)\to0$ as
$t\downarrow0$ such that
\begin{equation}\label{eq:metric-conversion}
 |\Psi(u)-\Psi(v)|\le\omega_A(t)
\end{equation}
whenever
\[
 d_{\HH/\Z}(u,v)\le A,\qquad
 0<\im u,\im v\le t.
\]
For sufficiently small $t$, one may take
\begin{equation}\label{eq:omega}
 \omega_A(t)=A\eta(e^At).
\end{equation}
\end{lemma}

\begin{proof}
The infimum defining the quotient distance is attained: for fixed
$u,v$, the distances $d_{\HH}(u,v+k)$ tend to infinity as $|k|\to\infty$.
Choose the corresponding integer translate of $v$. Periodicity leaves
its image under $\Psi$ unchanged.

Join the chosen lifts by a hyperbolic geodesic of length at most $A$.
Along any path parameterized by hyperbolic arclength,
$|d\log(\im w)|\le ds_{\HH}$. Thus this geodesic stays below height
$e^At$. Using Lemma~\ref{lem:little-bloch},
\begin{align*}
 |\Psi(u)-\Psi(v)|
 &\le\int|\Psi'(w)|\,|dw|\\
 &\le\eta(e^At)\int\frac{|dw|}{\im w}
 \le A\eta(e^At).
\end{align*}
\end{proof}

\begin{remark}
Neither local connectivity nor a boundary extension of $\Psi$ was used.
The quantity that vanishes is the area of a \emph{shrinking exterior
collar}, not the area of $K$. The argument therefore makes no
zero-area assumption about the hedgehog.
\end{remark}

\section{Osculating curves in the dynamical plane}\label{sec:physical}

Define
\begin{equation}\label{eq:physical-curve}
 \Gamma_n=\{\Psi(\gamma_n(x)):x\in\R/\Z\},
\end{equation}
and let $\Omega_n$ be its bounded Jordan interior.

\begin{lemma}[The enclosed domain]\label{lem:enclosed}
The curve $\Gamma_n$ is a real-analytic Jordan curve, $K\subset\Omega_n$,
and
\begin{equation}\label{eq:enclosed}
 \Omega_n\setminus K
 =\Psi\bigl(\{x+iy:0<y<c\ell_n(x)\}/\Z\bigr).
\end{equation}
The closed domains $\overline{\Omega_n}$ lie in every prescribed open
neighborhood of $K$ for all sufficiently large $n$.
\end{lemma}

\begin{proof}
Under $w\mapsto e^{-2\pi i w}$, the graph has one positive radius
$e^{2\pi c\ell_n(x)}$ for each angle $-2\pi x$. It is therefore a
real-analytic Jordan curve outside the unit circle. The exterior Riemann
map is univalent near this curve and preserves the two sides, with the
unbounded side containing infinity. The bounded side contains $K$,
and the annular region between $K$ and the curve is precisely
\eqref{eq:enclosed}.

Equivalently, the image graph is homotopic in $\C\setminus K$ to
an enclosing equipotential. Its winding number about every point of $K$
is nonzero, so $K$ lies in its Jordan interior.

Since $c\ell_n(x)\le cM_n\to0$, the last assertion follows from
Lemma~\ref{lem:collar}, including the graph boundary itself.
\end{proof}

\begin{proposition}[Holomorphy on the entire enclosed domain]\label{prop:long-iterates}
For all sufficiently large $n$, all maps $f^j$, $0\le j\le Q_n$, are
holomorphic on a neighborhood of $\overline{\Omega_n}$. All their values
on $\overline{\Omega_n}$ lie in one fixed bounded neighborhood of $K$.
For $z=\Psi(x+i s\ell_n(x))$, $0<s\le c$,
\begin{equation}\label{eq:physical-conjugacy}
 f^j(z)=\Psi\bigl(g^j(x+i s\ell_n(x))\bigr)
 \qquad(0\le j\le Q_n).
\end{equation}
\end{proposition}

\begin{proof}
Fix a bounded neighborhood $U_0$ of $K$ with
$\overline{U_0}\compactin U$, where $f$ is holomorphic and one-to-one.
By the collar lemma and \eqref{eq:strip-height}, the image under $\Psi$
of the strip $0<\im w\le t_n$ lies in $U_0$ for sufficiently large $n$.
Choose $n$ still larger if needed so this strip is contained in the
external coordinate domain of \eqref{eq:g-lift}.

Every point of $\overline{\Omega_n}\setminus K$ has the form used in
\eqref{eq:physical-conjugacy}, by Lemma~\ref{lem:enclosed}.
Proposition~\ref{prop:tube} shows that its successive external iterates
through time $Q_n$ remain in $0<\im w\le t_n$. Induction using
\eqref{eq:g-lift} proves \eqref{eq:physical-conjugacy} and places all
these iterates in $U_0$. Points of $K$ remain in $K$ by invariance.

To obtain a neighborhood of the whole closed domain, set $U^{(1)}=U$
and define recursively
\[
 U^{(r+1)}=\{z\in U^{(r)}:f^r(z)\in U\}.
\]
Each $U^{(r)}$ is open, and the first $r$ iterates are holomorphic there.
The preceding orbit estimates show that
$\overline{\Omega_n}\subset U^{(Q_n)}$. Hence every iterate under
consideration is holomorphic on an open neighborhood of
$\overline{\Omega_n}$. This neighborhood may depend on $n$.
\end{proof}

\begin{proposition}[Uniform Euclidean orbit matching]\label{prop:euclidean-match}
There are $\varepsilon_n\to0$ such that, for every
$\zeta,\xi\in\Gamma_n$, a permutation $\sigma$ of
$\{0,\ldots,Q_n-1\}$ satisfies
\begin{equation}\label{eq:euclidean-match}
 \max_{0\le j<Q_n}|f^j(\zeta)-f^{\sigma(j)}(\xi)|
 \le\varepsilon_n.
\end{equation}
One may take $\varepsilon_n=\omega_C(t_n)$, with $C$ from
\eqref{eq:Cmatch} and $\omega_C$ from Lemma~\ref{lem:metric-conversion}.
\end{proposition}

\begin{proof}
Write $\zeta=\Psi(\gamma_n(x))$ and
$\xi=\Psi(\gamma_n(y))$. Proposition~\ref{prop:hyperbolic-match} supplies
a permutation with hyperbolic distances at most $C$, and all endpoint
heights at most $t_n$. Apply Lemma~\ref{lem:metric-conversion} to each
matched pair, then use \eqref{eq:physical-conjugacy}.
\end{proof}

\section{Holomorphic averages and uniform recurrence}\label{sec:averages}

\begin{proposition}[Uniform holomorphic averages]\label{prop:poly-averages}
For every polynomial $P$,
\begin{equation}\label{eq:poly-limit}
 \norm{A_{Q_n}P-P(0)}_K\longrightarrow0.
\end{equation}
If $L_P=\max|P'|$ on a fixed disk containing these orbit points,
then
\begin{equation}\label{eq:poly-bound}
 \norm{A_{Q_n}P-P(0)}_K\le2L_P\omega_C(t_n).
\end{equation}
This proposition also holds in the linearizable case and does not
require empty interior.
\end{proposition}

\begin{proof}
Lemma~\ref{lem:enclosed} and
Propositions~\ref{prop:long-iterates}--\ref{prop:euclidean-match}
verify every hypothesis of Proposition~\ref{prop:criterion}.
Apply \eqref{eq:poly-two}.
\end{proof}

\begin{proposition}[Uniform closest-return recurrence]\label{prop:recurrence}
With the consecutive denominators of the external rotation number,
\begin{equation}\label{eq:uniform-recurrence}
 \norm{f^{q_n}-\id}_K\longrightarrow0.
\end{equation}
The same statement holds for $f^{-q_n}$ on $K$.
\end{proposition}

\begin{proof}
Proposition~\ref{prop:quasi} and
Lemma~\ref{lem:metric-conversion} give, on $\Gamma_n$,
\[
 |f^{q_n}(z)-z|\le\omega_{D+L}(t_n)\longrightarrow0.
\]
The map $f^{q_n}-\id$ is holomorphic on a neighborhood of
$\overline{\Omega_n}$ by Proposition~\ref{prop:long-iterates}, since
$q_n\le Q_n$. The maximum principle yields the same bound on $K$.
Since $f^{q_n}:K\to K$ is a bijection,
\[
 \sup_{z\in K}|f^{-q_n}(z)-z|
 =\sup_{w\in K}|w-f^{q_n}(w)|,
\]
which proves the inverse statement.
\end{proof}

This uniform recurrence is a classical consequence of quasi-invariant
curves \cite{PMQ,PMBB}. We use it below to study the interior of $K$.

\section{The interior of a Siegel compactum}\label{sec:interior}

It is known from \cite{PMQ} that the interior of a non-linearizable hedgehog is
empty. We prove it here for completeness.

\begin{lemma}[Components of the interior are simply connected]\label{lem:simply-connected}
Every connected component of $\Int K$ is simply connected.
\end{lemma}

\begin{proof}
Let $W$ be such a component and let $J\subset W$ be a Jordan curve.
Its bounded Jordan interior cannot meet $\C\setminus K$: otherwise
the connected unbounded open set $\C\setminus K$ would contain a path
from an interior point to infinity, crossing $J\subset K$.
Thus the Jordan interior is contained in $K$. Being open, it is contained
in $\Int K$; being connected and meeting $W$ near $J$, it is contained
in $W$. The Jordan-curve criterion for simple connectivity now applies.
\end{proof}

\begin{lemma}[All interior components are invariant]\label{lem:invariant-components}
Every connected component $W$ of $\Int K$ satisfies $f(W)=W$.
\end{lemma}

\begin{proof}
Because $f$ is a homeomorphism between neighborhoods of $K$ and maps
$K$ onto itself, it maps $\Int K$ onto itself and permutes its connected
components bijectively. Fix $z\in W$ and a small disk about $z$ contained
in $W$. By Proposition~\ref{prop:recurrence}, $f^{q_n}(z)\in W$ for
every sufficiently large $n$. Hence $f^{q_n}(W)=W$ for all those $n$.

The set of integers $m$ with $f^m(W)=W$ is a subgroup of $\Z$. It
contains a positive integer and therefore equals $p\Z$ for some $p\ge1$.
It contains two consecutive large denominators $q_n,q_{n+1}$, whose
greatest common divisor is $1$. Thus $p=1$ and $f(W)=W$.
\end{proof}

\begin{lemma}[Recurrence for a disk automorphism]\label{lem:disk-auto}
If an automorphism $G$ of $\D$ has a point $w\in\D$ and integers
$n_k\to\infty$ with $G^{n_k}(w)\to w$, then $G$ has a fixed point in
$\D$.
\end{lemma}

\begin{proof}
Recall that a disk automorphism is a fractional-linear transformation
$G(z)=e^{it}(z-a)/(1-\overline a z)$. If it has no fixed point in the
disk, it has a fixed point on the unit circle. One can see this directly
from the fixed-point quadratic: reflection in the unit circle pairs
non-boundary fixed points, so a non-boundary fixed point outside would
have a partner inside. The rotation case is immediate.

Send a boundary fixed point to infinity by a Cayley transform. The
conjugate automorphism of the upper half-plane fixes infinity and has
the form $w\mapsto Aw+B$, with $A>0$ and $B\in\R$. If $A\ne1$, after
a real translation it is $w\mapsto Aw$; its forward iterates converge
to a boundary point, either $0$ or infinity. If $A=1$ and $B\ne0$, its
iterates tend to infinity. The remaining case is the identity, which
does have interior fixed points. None of the nonidentity cases without
an interior fixed point allows a recurrent point in the upper
half-plane. This proves the lemma.
\end{proof}

\begin{proposition}[An interior component forces linearizability]\label{prop:empty}
If $\Int K\ne\varnothing$, then it has exactly one component, this
component contains $0$, and the germ of $f$ at $0$ is holomorphically
linearizable. Consequently, in the non-linearizable case,
\begin{equation}\label{eq:empty-interior}
 \Int K=\varnothing.
\end{equation}
\end{proposition}

\begin{proof}
Let $W$ be an interior component. It is bounded and simply connected,
and $f:W\to W$ is a conformal automorphism by
Lemmas~\ref{lem:simply-connected} and \ref{lem:invariant-components}.
A Riemann map conjugates this restriction to a disk automorphism.
Uniform recurrence on $K$ supplies an interior recurrent point of that
automorphism. Lemma~\ref{lem:disk-auto} therefore gives a fixed point
$a\in W$ of $f$.

Apply Proposition~\ref{prop:poly-averages} to the polynomial $P(z)=z$.
Since $f^j(a)=a$ for every $j$, its average at $a$ is exactly $a$,
whereas the proposition says this average tends to $P(0)=0$. Thus
$a=0$. Every interior component would therefore contain $0$, so there
is only one.

Choose a Riemann map $\chi:\D\to W$ with $\chi(0)=0$. The map
$\chi^{-1}\circ f\circ\chi$ is an automorphism of $\D$ fixing zero.
Schwarz's lemma applied to it and its inverse gives
\[
 \chi^{-1}\circ f\circ\chi(z)=e^{2\pi i\theta}z.
\]
This is a holomorphic linearization on $W$, in particular near $0$.
Its contrapositive proves \eqref{eq:empty-interior}.
\end{proof}

\begin{remark}
Only the polynomial $P(z)=z$ is used to locate the fixed point in the
preceding proof. Polynomial approximation is needed only to pass from
holomorphic averages to arbitrary continuous functions.
\end{remark}

\section{Unique ergodicity}\label{sec:ergodicity}

\begin{hypothesis}[Mergelyan approximation \cite{Mergelyan51}]\label{input:mergelyan}
If a compact set $E\subset\C$ has connected complement, every continuous
function on $E$ that is holomorphic on $\Int E$ can be approximated
uniformly on $E$ by complex polynomials. In particular, if
$\Int E=\varnothing$, the polynomials are dense in $C(E)$.
\end{hypothesis}
We use the theorem in its empty-interior case. For a proof of the
general statement, see \cite{Mergelyan51}.

From now on, assume that the germ is non-linearizable. Then $K$ is full
and has empty interior by Proposition~\ref{prop:empty}, so the last
clause of the approximation theorem applies.

\begin{proposition}[Continuous averages along denominators]\label{prop:continuous-subsequence}
For every $\varphi\in C(K)$,
\begin{equation}\label{eq:continuous-subsequence}
 \norm{A_{Q_n}\varphi-\varphi(0)}_K\longrightarrow0.
\end{equation}
\end{proposition}

\begin{proof}
Given $\epsilon>0$, choose a polynomial $P$ with
$\norm{\varphi-P}_K<\epsilon$. Invariance of $K$ gives
\[
 \norm{A_{Q_n}\varphi-A_{Q_n}P}_K\le\epsilon,
 \qquad |\varphi(0)-P(0)|\le\epsilon.
\]
By Proposition~\ref{prop:poly-averages},
\[
 \limsup_{n\to\infty}
 \norm{A_{Q_n}\varphi-\varphi(0)}_K\le2\epsilon.
\]
Letting $\epsilon\downarrow0$ proves the claim.
\end{proof}

\begin{proof}[Proof of the invariant-measure assertion in Theorem~\ref{thm:main}]
The probability $\delta_0$ is invariant since $f(0)=0$, so existence is
immediate. Let $\mu\in\Mf(K)$ be arbitrary. For every continuous
$\varphi$, invariance gives
\[
 \int_K A_{Q_n}\varphi\dd\mu=\int_K\varphi\dd\mu.
\]
Using uniform convergence in
Proposition~\ref{prop:continuous-subsequence}, we obtain
\[
 \int_K\varphi\dd\mu=\varphi(0).
\]
This identifies $\mu$ with $\delta_0$. Alternatively, applying the same
argument to $\varphi(z)=|z|^2$ gives
$\int|z|^2\dd\mu=0$, which already forces the probability $\mu$ to be
concentrated at $0$. This proves \eqref{eq:unique}.
\end{proof}

\subsection{From selected denominators to every averaging time}
It remains to pass from $Q_n$ to all averaging times. Uniformity in
the starting point allows us to decompose a long orbit segment into
blocks of one fixed good length.

\begin{lemma}[Uniform blocking]\label{lem:blocking}
Let $T:X\to X$ be a map of a set and let $\varphi$ be bounded. Suppose
that for integers $Q_k\to\infty$ and a constant $b$,
\[
 \sup_{x\in X}\left|\frac1{Q_k}\sum_{j=0}^{Q_k-1}\varphi(T^jx)-b\right|
 \longrightarrow0.
\]
Then the analogous uniform convergence holds as the averaging length
$N$ tends to infinity through all positive integers.
\end{lemma}

\begin{proof}
Fix $\epsilon>0$ and choose a single $Q=Q_k$ such that every length-$Q$
average differs from $b$ by at most $\epsilon$. Write $N=mQ+r$, with
$0\le r<Q$. Splitting the sum into $m$ complete blocks and one remainder
shows
\[
 \sup_{x\in X}\left|\frac1N\sum_{j=0}^{N-1}\varphi(T^jx)-b\right|
 \le\epsilon+\frac{Q}{N}(\norm{\varphi}_\infty+|b|).
\]
Every complete block starts at another point of $X$, so the same uniform
bound applies to it. First let $N\to\infty$ with $Q$ fixed and then
let $\epsilon\downarrow0$.
\end{proof}

Apply this lemma with $X=K$, $T=f$, and $b=\varphi(0)$ to prove
\eqref{eq:uniform-main}. A good denominator is fixed before $N$ tends to infinity; no bound
on $Q_{n+1}/Q_n$ is required.

\subsection{Uniformly sublinear excursions}
For $\varphi(z)=|z|^2$, \eqref{eq:uniform-main} becomes
\begin{equation}\label{eq:second-moment}
 \lim_{N\to\infty}\sup_{z\in K}\frac1N
       \sum_{j=0}^{N-1}|f^j(z)|^2=0.
\end{equation}
For each $r>0$,
\[
 r^2\#\{0\le j<N:|f^j(z)|\ge r\}
 \le\sum_{j=0}^{N-1}|f^j(z)|^2.
\]
Take the supremum and use \eqref{eq:second-moment} to obtain
\eqref{eq:excursion-main}. This completes the proof of
Theorem~\ref{thm:main}.

The conclusion concerns the frequency of excursions, not pointwise
convergence to $0$. Excursions of density zero are compatible with the
uniform recurrence proved above.

\section{Quantitative estimates}\label{sec:audit}

For a polynomial $P$, combining \eqref{eq:poly-bound},
\eqref{eq:omega}, and \eqref{eq:little-bloch-bound} gives
\begin{equation}\label{eq:explicit-error}
 \norm{A_{Q_n}P-P(0)}_K
 \le2L_PC\sqrt{\frac8\pi\mathcal E(2e^Ct_n)},
 \qquad t_n=2cBM_n.
\end{equation}
Here $\mathcal E(a)$ is the area of the exterior collar defined in
\eqref{eq:finite-area}. Since $M_n\to0$ and $\mathcal E(a)\to0$ as
$a\downarrow0$, the right-hand side tends to zero.

For any continuous $\varphi$ and any polynomial $P$,
\begin{equation}\label{eq:approx-error}
 \norm{A_{Q_n}\varphi-\varphi(0)}_K
 \le2\norm{\varphi-P}_K+
       2L_PC\sqrt{\frac8\pi\mathcal E(2e^Ct_n)}.
\end{equation}
The polynomial is fixed before taking $n\to\infty$. These estimates
give a modulus of convergence for the particular compactum, not a
universal rate in $Q_n$.

\begin{remark}[Choice of constants]
Fix $f$, $K$, and the analytic strip for the external map. This fixes
$V,B,\Lambda,\tau$. Choose $c$ as in \eqref{eq:small-c}; the constants
$D,L,C$ are then fixed. Only afterwards is $n$ taken large. In
particular, $e^Ct_n\to0$, however large the fixed constant $C$ is.
No Diophantine estimate or regularity estimate for the conjugacy $H$
is required.
\end{remark}

\section{Further results}\label{sec:further}

The following results will be proved in a forthcoming article.

\medskip

\textbf{Boundaries of Siegel disks.}

\medskip

The results extend to a linearizable hedgehog in the following sense: The
dynamics restricted to the boundary of the Siegel disk of a linearizable hedgehog
is uniquely ergodic. The
invariant measure is the harmonic measure in the boundary from the fixed point.

\medskip
\textbf{Annular hedgehogs.}

\medskip

An annular hedgehog is associated with a non-linearizable analytic circle diffeomorphism
with irrational rotation number. The dynamics on an annular hedgehog is uniquely
ergodic, the invariant measure being the unique invariant measure on the circle.

\medskip
\textbf{Essential compacta.}

\medskip

We consider an essential compact connected set $K\subset\C^*$ such that
$\C^*\setminus K$ has exactly two connected components, $U_0$ and $U_\infty$,
with $U_0\cup\{0\}$ and $U_\infty\cup\{\infty\}$ simply connected.
We assume that the dynamics $f$ and its inverse are defined in an essential annular
neighborhood of $K$. Then the prime-end rotation numbers of the two complementary
components agree, with compatible orientation conventions, and are supposed to be irrational. The dynamics on $K$ is uniquely
ergodic if and only if $K$ does not contain a linearizable Herman ring.

\bigskip

\textbf{Note to Button Pushers and AI companies.}

\medskip

Please, don't spam proposing computer generated proofs of these results that
you don't understand.
I already  have generated these proofs but my goal is to write beautiful
human proofs. Thank you for your attention.

\end{document}